\documentclass[11pt]{article}
\usepackage[a4paper, total={6in, 8in}]{geometry}
\usepackage{enumerate}
\usepackage{amsmath}
\usepackage{amssymb,latexsym}
\usepackage{amsthm}
\usepackage{color}
\usepackage{cancel}
\usepackage{graphicx}
\usepackage{cite}
\usepackage{changes}
\usepackage{lineno}
\usepackage{hyperref}
\usepackage{comment}
\usepackage{amscd}

\newtheorem{theorem}{Theorem}[section]

\newtheorem{lemma}[theorem]{Lemma}
\newtheorem{corollary}[theorem]{Corollary}
\newtheorem{definition}[theorem]{Definition}
\newtheorem{proposition}[theorem]{Proposition}

\newtheorem{example}[theorem]{Example}

\newtheorem{remark}[theorem]{Remark}

\newcommand{\fqn}{\mathbb{F}_{q^n}}

\newcommand{\F}{{\mathbb F}}

\newcommand{\fq}{{\mathbb F}_{q}}

\newcommand{\PG}{\mathrm{PG}}

\title{Classifications and constructions of minimum size linear sets}
\author{Vito Napolitano, Olga Polverino, Paolo Santonastaso and 
Ferdinando Zullo\thanks{Dipartimento di Matematica e Fisica, Universit\`a degli Studi della Campania ``Luigi Vanvitelli'', Caserta, Italy.
\{vito.napolitano,olga.polverino,paolo.santonastaso,ferdinando.zullo\}@unicampania.it}
}
\date{ }

\begin{document}
\maketitle

\begin{abstract}
This paper aims to study linear sets of \emph{minimum size} in the projective line, that is $\fq$-linear sets of rank $k$ in $\PG(1,q^n)$ admitting one point of weight one and having size $q^{k-1}+1$. Examples of these linear sets have been found by Lunardon and the second author (2000) and, more recently, by Jena and Van de Voorde (2021).
However, classification results for minimum size linear sets are known only for $k\leq 5$. In this paper we provide classification results for those $L_U$ admitting two points with complementary weights.
We construct new examples and also study the related $\mathrm{\Gamma L}(2,q^n)$-equivalence issue.
These results solve an open problem posed by Jena and Van de Voorde.
The main tool relies on two results by Bachoc, Serra and Z\'emor (2017 and 2018) on the linear analogues of Kneser's and Vosper's theorems.
We then conclude the paper pointing out a connection between critical pairs and linear sets, obtaining also some classification results for critical pairs.
\end{abstract}


\section{Introduction}

Let $q$ be a prime power, $n$ be a positive integer and let $\fqn$ be the degree $n$ extension of $\fq$. Let $V$ be an $r$-dimensional $\fqn$-vector space.
An $\fq$-\textbf{linear set} in $\PG(r-1,q^n)=\PG(V,\fqn)$ can be defined as a set of points for which their defining vectors lie in a fixed $\fq$-subspace of $V$. The term \emph{linear} dates back to \cite{Lun}, where Lunardon constructs special type of blocking sets. Since then, linear sets have been used to construct and to classify several algebraic and geometric objects, and more recently they have been also employed in coding theory. We refer to \cite{LavVdV,Polverino} for more details.

This paper aims to study $\fq$-linear set in $\PG(1,q^n)$ having \emph{minimum size}. 
As proved by De Beule and Van de Voorde in \cite[Theorem 1.2]{DeBeuleVdV} (generalizing \cite[Lemma 2.2]{BoPol}), an $\fq$-linear set $L_U$ of rank $k$ in $\PG(1,q^n)$ with at least one point point of weight one has at least $q^{k-1}+1$ points.
A such linear set with exactly this number of points is said to be of \textbf{minimum size}.
The classical example is $L_U$ where
\[ U=\{ (x,\mathrm{Tr}_{q^n/q}(x)) \colon x \in \fqn\}, \]
and $\mathrm{Tr}_{q^n/q}(x)=x+\ldots+x^{q^{n-1}}$, which has size $q^{n-1}+1$, $q^{n-1}$ points of weight one and one point of weight $n-1$.
Other examples of such linear sets with maximum rank $n$ have been found in \cite{LunPol2000}, recently generalized by Jena and Van de Voorde \cite{DVdV}.
Minimum size linear sets in $\PG(1,q^n)$ of rank at most $5$ have been classified and they all correspond to the linear sets described above.
Jena and Van de Voorde in \cite[(D) Section 2.5]{DVdV} left as an open problem the determination of whether or not there exist other examples of minimum size linear sets.
In this paper, we analyze the case of a minimum size linear set $L_U$ of rank $k$ in $\PG(1,q^n)$ admitting two points of \emph{complementary weights of type} $(k-r,r)$, that is there exist two distinct points $P_1$ and $P_2$ of weight $k-r$ and $r$ (with $k-r\geq r$) in $L_U$, respectively.
We first provide classification results of minimum size linear sets $L_U$ in the following cases:
\begin{itemize}
    \item[(A)] $L_U$ admits two points with complementary weights of type $(k-2,2)$;
    \item[(B)] $L_U$ admits two points with complementary weights of type $(k-r,r)$ and $n$ is prime.
\end{itemize}
In this latter case, we prove that they all correspond to the construction found in \cite{DVdV}, whereas in the former case we provide more examples.
Then we introduce a general procedure to construct examples of linear sets in $\PG(1,q^n)$ starting from a linear set $L_{U'}$ in $\PG(1,q^t)$, with $t \mid n$, whose weight spectrum is completely determined by the weight distribution of $L_{U'}$.
In particular, using the family of minimum size linear sets of \cite{DVdV}, we are able to get more examples of minimum size linear sets different from those found in \cite{DVdV} and also with a different \emph{weight spectrum}. We then study the $\mathrm{\Gamma L}(2,q^n)$-equivalence among the subspaces defining the linear set of \cite{DVdV} and those found in this paper.
These results answer to the open problem posed in \cite[(D) Section 2.5]{DVdV}.

The main tools in our investigation regard the linear analogue of Kneser's theorem and Vosper's theorem from additive combinatorics, established both by Bachoc, Serra and Z\'emor.
We will recall these results in the framework of finite fields extensions and we use the following notation.
Let $S$ and $T$ be two $\fq$-subspaces of $\fqn$, then $ST=\{s\cdot t \colon s \in S, t\in T\}$.

\begin{theorem} \cite[Theorem 3]{BSZ2015} \label{teo:bachocserrazemor}
Let $S$ be an $\F_q$-subspace of $\F_{q^n}$. Then
\begin{itemize}
    \item either for every $\F_q$-subspace $T \subseteq \F_{q^n}$ we have
    \[
    \dim_{\F_q}(\langle S T \rangle_{\F_q})\geq \min\{\dim_{\F_q}(S)+\dim_{\F_q}(T)-1,n\},
    \]
    \item or there exists a positive integer $t>1$ dividing $n$, such that for every $\F_q$-subspace $T \subseteq \F_{q^n}$ satisfying 
    \[
    \dim_{\F_q}(\langle S T \rangle_{\F_q})< \dim_{\F_q}(S)+\dim_{\F_q}(T)-1,
    \]
    we have that $\langle S T \rangle_{\F_q}$ is also an $\F_{q^t}$-subspace.
\end{itemize}
\end{theorem}

\begin{theorem} \cite[Theorem 3]{BSZ2017} \label{teo:bachocserrazemor2}
Suppose that $n$ is a prime. Let $S,T$ be $\F_q$-subspaces of $\F_{q^n}$ such that $2 \leq \dim_{\F_q}(S),\dim_{\F_q}(T)$ and $\dim_{\F_q}(\langle S T\rangle_{\fq}) \leq n-2$. If
\[
\dim_{\F_q}(\langle S T\rangle_{\fq}) =\dim_{\F_q}(S)+\dim_{\F_q}(T)-1,
\]
then $S=g \langle 1,a,\ldots,a^{\dim_{\F_q}(S)-1}\rangle_{\F_q}$ and $T=g' \langle 1,a,\ldots,a^{\dim_{\F_q}(T)-1}\rangle_{\F_q}$, for some $g,g',a \in \F_{q^n}$.
\end{theorem}

The first point of Theorem \ref{teo:bachocserrazemor} can be also seen as the Cauchy-Davenport analogue and, for this reason, we will call \textbf{critical pairs} the couples $(S,T)$ satisfying the equality in the first point of Theorem \ref{teo:bachocserrazemor} with \[\dim_{\F_q}(\langle S T \rangle_{\F_q})= \dim_{\F_q}(S)+\dim_{\F_q}(T)-1.\]
So, Theorem \ref{teo:bachocserrazemor2} classifies the critical pairs under the assumptions that $\dim_{\F_q}(\langle S T\rangle_{\fq}) \leq n-2$ and $n$ is prime.
These results allow us to prove the classification result on minimum size linear sets in Case (B).

However, we point out that the critical pairs of subspaces are related to linear sets admitting a certain number of points with a fixed weight. When $n$ is not prime, this connection allows us to present new examples of critical pairs of subspaces for finite fields extensions which are different from those in Theorem \ref{teo:bachocserrazemor2}, arising from the new examples of minimum size linear sets.

Finally, we conclude the paper by giving a classification result for critical pairs in extension fields $L/F$ in which one of the subspaces has dimension two (see Proposition \ref{prop:extalg}).

\section{Preliminaries}

We recall some definitions and results from linear sets theory and some properties of the trace function.

\subsection{Linear sets}

In this paper we mainly consider $\fq$-linear sets in the projective line, so in this section we recall some useful definitions and results regarding linear sets in $\PG(1,q^n)$.
Let $V$ be a $2$-dimensional $\fqn$-vector space and let $\Lambda=\PG(V,\fqn)=\PG(1,q^n)$. 
Let $U \neq \{0\}$ be an $\fq$-subspace of $V$, then the set
\[ L_U=\{\langle \mathbf{u} \rangle_{\fqn} \colon \mathbf{u}\in U\setminus\{\mathbf{0}\}\} \]
is said to be an $\fq$-\textbf{linear set} of rank $\dim_{\fq}(U)$.
The rank of $L_U$ will also be denoted by $\mathrm{Rank}(L_U)$.

Regarding the linearity of a linear set, more attention has to be paid as pointed out by Jena and Van de Voorde in \cite{DVdV2}.

\begin{definition}\cite[Definition 1.1]{DVdV2}
An $\F_{q}$-linear set, defined by an  $\F_{q}$--vector space $U$ is  an \textbf{$\F_{q^s}$-linear set} if $U$ is also an $\F_{q^s}$-vector space.
\end{definition}

\begin{definition}\cite[Definition 1.2]{DVdV2}
A \textbf{strictly} $\F_{q^s}$-linear set $L_U$ is an $\F_{q^s}$-linear set such that $U$ is not an $\F_{q^i}$-subspace of $V$ for any $i>s$. 
\end{definition}

The maximum field of linearity of a strictly $\F_{q^s}$-linear set is $\F_{q^s}$; see \cite{CsMP}.

\begin{definition}\cite[Definition 1.2]{DVdV2}
An $\F_q$-linear set $L_U$ has \textbf{geometric field of linearity} $\F_{q^s}$ if there exists an $\F_{q^s}$-linear set $L_W$ with $L_U=L_W$.
\end{definition}

Another important notion is the following. The \textbf{weight} of a point $P=\langle \mathbf{v}\rangle_{\fqn} \in \Lambda$ in $L_U$ is defined as $\dim_{\fq}(U \cap \langle \mathbf{v}\rangle_{\fqn})$.
Denote by $N_i$ the number of points of $\Lambda$ having weight $i\in \{0,\ldots,k\}$ in $L_U$.
We say that an $\fq$-linear set $L_U$ of rank $k$ has \textbf{weight spectrum} (with respect to $U$) $(i_1,\ldots,i_t)$ with $1\leq i_1 <i_2 < \ldots <i_t\leq k$ if for every $P \in L_U$ 
\[ w_{L_U}(P) \in \{i_1,\ldots,i_t\} \]
and each of these integers $i_j$ occurs as the weight of at least one point of $L_U$.
The \textbf{weight distribution} of $L_U$ (with respect to $U$) is $(N_{i_1},\ldots,N_{i_t})$.

\begin{remark} \label{rk:inequivalent}
It is worth to mention that if $U$ and $W$ are two $\fq$-subspaces of $V$ such that $L_U$ and $L_W$ have a distinct weight spectrum then $U$ and $W$ cannot be $\mathrm{\Gamma L}(2,q^n)$-equivalent.
\end{remark}

The $N_i$'s satisfy the following relations:
\begin{equation}\label{eq:card}
    |L_U| \leq \frac{q^k-1}{q-1},
\end{equation}
\begin{equation}\label{eq:pesicard}
    |L_U| =N_1+\ldots+N_k,
\end{equation}
\begin{equation}\label{eq:pesivett}
    N_1+N_2(q+1)+\ldots+N_k(q^{k-1}+\ldots+q+1)=q^{k-1}+\ldots+q+1.
\end{equation}
Moreover, the following holds
\begin{equation}\label{eq:wpointsrank}
    w_{L_U}(P)+w_{L_U}(Q)\leq \mathrm{Rank}(L_U),
\end{equation}
for any $P,Q \in \mathrm{PG}(1,q^n)$ with $P\ne Q$.

Linear sets attaining the bound in \eqref{eq:card} are called \textbf{scattered}, originally introduced in \cite{BL2000}.
An $i$-\textbf{club of rank $k$} in $\PG(1,q^n)$ is an $\F_{q}$-linear set of rank $k$ in $\PG(1,q^n)$ such that one point has weight $i$ and all the others have weight one. If $k=n$, then we will simply call it an $i$-club.
By \eqref{eq:pesicard} and \eqref{eq:pesivett}, we have that an $i$-club of rank $k$ has size $q^{k-1}+\ldots+q^i+1$. These linear sets are associated with special type of arcs known as \textbf{KM-arcs} introduced by Korchm\'aros and Mazzocca in \cite{KM}; see \cite[Theorem 2.1]{DeBoeckVdV2016} for the connection.

As proved in \cite[Theorem 1.2]{DeBeuleVdV} (and for $k=n$ in \cite[Lemma 2.2]{BoPol}), an $\fq$-linear set $L_U$ of rank $k$ in $\PG(1,q^n)$ with at least one point of weight one has at least $q^{k-1}+1$ points.
A such linear set with exactly this number of points is said to have \textbf{minimum size}.
Examples have been recently found by Jena and Van de Voorde, extending the examples in \cite{BoPol} and in \cite{LunPol2000}.

\begin{theorem}\cite[Theorem 2.7]{DVdV}\label{th:constructionVdV}
Let $\lambda\in \fqn \setminus \fq$ be an element generating a degree $s$-extension of $\fq$ and 
\[ L=\{ \langle (\alpha_0+\alpha_1 \lambda+\ldots+\alpha_{t_1-1}\lambda^{t_1-1},\beta_0+\beta_1 \lambda+\ldots+\beta_{t_2-1}\lambda^{t_2-1}) \rangle_{\fqn} \colon \alpha_i,\beta_i \in \fq,\,\]\[\text{not all zero},\, 1 \leq t_1,t_2, t_1+t_2 \leq s+1   \}. \]
Then $L$ is an $\fq$-linear set of $\PG(1,q^n)$ of rank $k=t_1+t_2$ with $q^{k-1}+1$ points.
Let $t_1\leq t_2$, then 
\begin{itemize}
\item the point $\langle (0,1)\rangle_{\fqn}$ has weight $t_2$;
\item there are $q^{t_2-t_1+1}$ points of weight $t_1$ different from $\langle (0,1)\rangle_{\fqn}$;
\item there are $q^{k-2i+1}-q^{k-2i-1}$ points of weight $i \in \{1,\ldots, t_1-1\}$.
\end{itemize}
\end{theorem}

The examples of linear sets in Theorem \ref{th:constructionVdV} admits two points with complementary weights, that is the sum of the weights of such points equals the rank of the linear set. These type of linear sets have been recently investigated in \cite{NPSZ2021}. 
We recall a useful result from \cite{NPSZ2021}, regarding the number of points of  a linear set with two points with complementary weights.

\begin{proposition}\cite[Proposition 3.1]{NPSZ2021}
Let $L_U$ be an $\fq$-linear set with two points with complementary weights, namely $k-r$ and $r$. If there exists one point of weight one in $L_U$ then
\begin{equation}\label{eq:bound}
  q^{k-1}+1  \leq |L_U|\leq q^{k-1}+\ldots+q^{\max\{k-r,r\}}-q^{\min\{k-r,r\}-1}-\ldots-q+1.
\end{equation}
\end{proposition}

A linear set $L_U$ with complementary weights $k-r$ and $r$,  with $k-r \geq r$, will be said of \textbf{type} $(k-r,r)$. 
When the upper bound in \eqref{eq:bound} is attained the corresponding linear sets have interesting associated codes; see \cite{NPSZ2021code}.
In this paper we deal with linear sets attaining the lower bound in \eqref{eq:bound}.

We refer to \cite{LavVdV} and \cite{Polverino} for comprehensive references on linear sets.

\subsection{Trace properties}

In this subsection we recall some properties of trace function which will be very important to study the equivalence among $\fq$-subspaces and for the connection between linear sets and critical pairs.
The trace function from $\fqn$ over $\fq$ is defined as
\[
\mathrm{Tr}_{q^n/q}: a \in \F_{q^n} \mapsto \sum_{i=0}^{n-1} a^{q^i} \in \F_q,
\]
which is a linear map and it also defines a nondegenerate symmetric bilinear form as follows:
\[
(a,b) \in \F_{q^n} \times \F_{q^n} \mapsto \mathrm{Tr}_{q^n/q}(ab) \in \F_q.
\]

So, for any subset $S$ of $\fqn$ we can define the orthogonal complement as
\[ S^\perp=\{a \in \fqn \colon \mathrm{Tr}_{q^n/q}(ab)=0,\,\,\,\forall b \in S \}. \]

Two ordered $\F_{q}$-bases $\mathcal{B}=(\xi_0,\ldots,\xi_{n-1})$ and $\mathcal{B}^*=(\xi_0^*,\ldots,\xi_{n-1}^*)$ of $\F_{q^n}$ are said to be \textbf{dual bases} if $\mathrm{Tr}_{q^n/q}(\xi_i \xi_j^*)= \delta_{ij}$, for $i,j\in\{0,\ldots,n-1\}$, where $\delta_{ij}$ is the Kronecker symbol. It is well known that for any $\F_q$-basis $\mathcal{B}=(\xi_0,\ldots,\xi_{n-1})$ there exists a unique dual basis $\mathcal{B}^*=(\xi_0^*,\ldots,\xi_{n-1}^*)$ of $\mathcal{B}$, see e.g.\ \cite[Definition 2.30]{lidl_finite_1997}. 

\begin{lemma}\label{cor:dualbasis}\cite[Corollary 2.7]{NPSZ2021}
Let $\lambda \in \fqn$ such that $\mathcal{B}=(1,\lambda,\ldots,\lambda^{n-1})$ is an ordered $\fq$-basis of $\fqn$.
Let $f(x)=a_0+a_1x+\ldots+a_{n-1}x^{n-1}+x^n$ be the minimal polynomial of $\lambda$ over $\fq$. 
Then the dual basis $\mathcal{B}^*$ of $\mathcal{B}$ is 
\[ \mathcal{B}^*=(\delta^{-1}\gamma_0,\ldots,\delta^{-1}\gamma_{n-1}), \]
where $\delta=f'(\lambda)$ and $\gamma_i=\sum_{j=1}^{n-i} \lambda^{j-1}a_{i+j}$, for every $i \in \{0,\ldots,n-1\}$. 
\end{lemma}

\begin{remark}\label{rk:span}
By the above lemma, we also have the following equalities:
\[ \langle \delta^{-1} \gamma_{\ell},\ldots, \delta^{-1} \gamma_{n-1} \rangle_{\fq}= \delta^{-1} \langle 1,\ldots, \lambda^{n-\ell-1} \rangle_{\fq}, \]
for any $\ell$.
\end{remark}

As a consequence of Lemma \ref{cor:dualbasis}, we may prove that dual of an $\fq$-subspace in $\F_{q^n}$ generated by consecutive powers of a generator $\lambda$ of $\F_{q^n}$ over $\F_q$ is spanned by consecutive powers of $\lambda$, up to a nonzero scalar in $\F_{q^n}$, as well.

\begin{proposition} \label{prop:dualwithdual}
Let $\lambda \in \fqn$ such that $\mathcal{B}=(1,\lambda,\ldots,\lambda^{n-1})$ is an ordered $\fq$-basis of $\fqn$. Let $f(x)=a_0+a_1x+\ldots+a_{n-1}x^{n-1}+x^n$ be the minimal polynomial of $\lambda$ over $\fq$. Let $W=\langle 1,\lambda,\ldots,\lambda^{\ell-1} \rangle_{\F_q}$, with $\ell \in \{1,\ldots,n-1\}$. Then $ W^\perp=\delta^{-1}\langle 1,\lambda,\ldots,\lambda^{n-\ell-1} \rangle_{\F_q}$, where $\delta =f'(\lambda)$. 
\end{proposition}
\begin{proof} 
By Lemma \ref{cor:dualbasis}, the dual basis of $\mathcal{B}$ is $\mathcal{B}^*=(\delta^{-1}\gamma_0,\ldots,\delta^{-1}\gamma_{n-1})$, where $\delta=f'(\lambda)$ and $\gamma_i=\sum_{h=1}^{n-i} \lambda^{h-1}a_{i+h}$, for every $i \in \{0,\ldots,n-1\}$. In particular, we have that $\mathrm{Tr}_{q^n/q}(\lambda^i \delta^{-1}\gamma_h)=0$ for any $i \leq \ell-1$ and any $h \geq \ell$. This implies that $\delta^{-1}\gamma_h \in W^{\perp}$ and so, by Remark \ref{rk:span} we have $\delta^{-1}\langle 1,\lambda,\ldots,\lambda^{n-\ell-1}\rangle_{\F_q} =\delta^{-1}\langle \gamma_{\ell},\ldots,\gamma_{n-1}\rangle_{\F_q} \subseteq W^{\perp}$. Since $\dim_{\F_q}(W^{\perp})=n-\ell$, we have the desired equality. 
\end{proof}

\section{Classification of linear sets of minimum size}

In this section, we are going to study linear sets with two points with complementary weights. Let us start with the following auxiliary lemma, which extends \cite[Lemma 4]{BSZ2017} for finite fields extension.

\begin{lemma}\label{lemma:power}
Let $S$ be an $\fq$-subspace of $\fqn$ of dimension $k\geq2$ and let $\mu \in \fqn\setminus\fq$. Let $t=\dim_{\fq}(\fq(\mu))$.
\begin{itemize}
    \item [(a)] If $\dim_{\fq}(S\cap \mu S)=k$, then $S$ is an $\fq(\mu)$-subspace.
    \item [(b)] Suppose that $\dim_{\fq}(S\cap \mu S)=k-1$ and $t\geq k$. Then $S=b \langle 1,\mu,\ldots,\mu^{k-1}\rangle_{\fq}$, for some $b \in \fqn^*$ and $t \neq k$.
    \item[(c)] Suppose that $\dim_{\fq}(S\cap \mu S)=k-1$ and $t\leq k-1$. Write $k=t\ell+m$ with $m<t$, then $m>0$ and $S=\overline{S}\oplus b\langle 1,\mu,\ldots,\mu^{m-1}\rangle_{\fq}$, where $\overline{S}$ is an $\F_{q^t}$-subspace of dimension $\ell$, $b \in \fqn^*$ and $b \F_{q^t} \cap \overline{S}=\{0\}$.
    In particular, $t$ is a proper divisor of $n$.
\end{itemize}
\end{lemma}
\begin{proof}
\textbf{(a)} Suppose that $\mu S=S$. Clearly, for any $s \in S\setminus\{0\}$, $\mu(s^{-1} S)=s^{-1}S$. So, up to multiply by the inverse of some element in $S$, we may assume that $1 \in S$. So that $\mu \in S$ and hence all the powers of $\mu$ are in $S$.
This implies that $\fq(\mu)\subseteq S$. If  $\alpha \in \fq(\mu)$ then $\alpha=c_0+c_1\mu+\ldots+c_{t-1}\mu^{t-1}$ for some $c_0,\ldots,c_{t-1} \in \fq$.
Hence, $\alpha s \in S$ for every $\alpha \in \fq(\mu)$ and $s \in S$, and so $S$ in an $\fq(\mu)$-subspace.\\
\textbf{(b)} Assume that $t\geq k$ and $\dim_{\fq}(S\cap \mu S)=k-1$. Consider  
\[S\cap \mu S\cap \mu^2 S=(S\cap \mu S)\cap \mu (S\cap \mu S)\subseteq S \cap \mu S,\] 
then $\dim_{\fq}(S\cap \mu S\cap \mu^2 S)\in \{ k-2,k-1\}$, since $S\cap \mu S$ and $\mu(S\cap \mu S)$ are both contained in $\mu S$.
Moreover, if $\dim_{\fq}(S\cap \mu S\cap \mu^2 S)= k-1$, then by (a) it follows that $S\cap \mu S$ is an $\F_{q^t}$-subspace of $\fqn$ and hence $k-1 \geq t$, a contradiction to the assumption that $t\geq k$.
So $\dim_{\fq}(S\cap \mu S\cap \mu^2 S)= k-2$. Arguing as above, we have 
\[ \dim_{\fq}(S\cap \ldots \cap \mu^{i} S)=k-i, \]
for any $i \in \{1,\ldots, k\}$.
In particular, $\dim_{\fq}(S\cap \ldots \cap \mu^{k-1} S)=1$. Thus, there exists a nonzero $a_0 \in S\cap \ldots \cap \mu^{k-1} S$ for which
\[ \mu^{k-1}a_{k-1}=\ldots=\mu a_1=a_0, \]
for some $a_{k-1},\ldots,a_1 \in S\setminus\{0\}$.
Therefore, $a_i=\mu^{k-1-i}a_{k-1}$ for any $i \in \{0,\ldots,k-2\}$. Hence,  $a_{k-1},a_{k-1}\mu,\ldots,a_{k-1}\mu^{k-1} \in S$ and they are $\fq$-linearly independent due to the fact that $t\geq k$, and  so 
\[S=\langle a_{k-1},a_{k-1}\mu,\ldots,a_{k-1}\mu^{k-1} \rangle_{\fq}=a_{k-1}\langle 1,\mu,\ldots,\mu^{k-1} \rangle_{\fq}.\]
Note that $t \neq k$, otherwise $S=a_{k-1}\F_{q^t}$ and so $S = \mu S$, a contradiction.\\
\noindent \textbf{(c)} Suppose now that $t<k$ and $\dim_{\fq}(S\cap \mu S)=k-1$. Consider $S\cap \mu S\cap \mu^2 S=(S\cap \mu S)\cap \mu (S\cap \mu S)\subseteq S \cap \mu S$. As in the previous case we have that $\dim_{\fq}(S\cap \mu S\cap \mu^2 S)\in \{k-2,k-1\}$.
If $\dim_{\fq}(S\cap \mu S\cap \mu^2 S)= k-1$ then $S\cap \mu S=\mu(S\cap \mu S)$. By (a) it follows that $S\cap \mu S$ is an $\F_{q^t}$-subspace and denote it by $\overline{S}$. In particular, $t \mid k-1$ and so $k=t\ell+1$. Let $s \in S \setminus \overline{S}$, then $S=\overline{S}\oplus \langle s\rangle_{\fq}$ and since $s\F_{q^t}\cap \overline{S}=\{0\}$, in this case the assertion is proved.
If $\dim_{\fq}(S\cap \mu S\cap \mu^2 S)= k-2$, then consider \[S\cap \mu S\cap \mu^2 S \cap \mu^3 S =(S\cap \mu S\cap \mu^2 S)\cap \mu(S\cap \mu S\cap \mu^2 S)\subseteq S\cap \mu S\cap \mu^2 S. \]
We have that $\dim_{\F_q}(S\cap \mu S\cap \mu^2 S \cap \mu^3 S) \in \{k-3,k-2\}$. We proceed by finite induction.
So, let $j$ be the minimum positive integer such that
\[ \dim_{\fq}(S\cap \ldots \cap \mu^{j} S)=k-j \]
and 
\[ \dim_{\fq}(S\cap \ldots \cap \mu^{j+1} S)=k-j. \]
In particular, 
\[ \dim_{\fq}(S\cap \ldots \cap \mu^{j-1} S)=k-j+1. \]
Let $\overline{S}=S\cap \ldots \cap \mu^{j} S$. Clearly, $\overline{S}=\mu \overline{S}$ and hence by (a) $\overline{S}$ is an $\F_{q^t}$-subspace. 
Since $\dim_{\fq}(\overline{S})=k-j$ and $\overline{S}$ is an $\F_{q^t}$-subspace, it follows that $t \mid k-j$, that is $k-j=\ell t$, for some non-negative integer $\ell$.
Since $S\cap \ldots \cap \mu^{j-1} S\supset \overline{S}$, there exists $a_0 \in (S\cap \ldots \cap \mu^{j-1} S) \setminus \overline{S}$ such that 
\[ S\cap \ldots \cap \mu^{j-1} S= \overline{S}\oplus \langle a_0\rangle_{\fq}. \]
Since $a_0 \in S\cap \ldots \cap \mu^{j-1} S$, then 
\begin{equation}\label{eq:ais} 
\mu^{j-1}a_{j-1}=\ldots=\mu a_1=a_0, 
\end{equation}
for some $a_{j-1},\ldots,a_1 \in S$. By \eqref{eq:ais} it follows that $a_i=\mu^{j-1-i} a_{j-1}$ for any $i \in \{0,\ldots,j-2\}$.
In particular, $a_{j-1},a_{j-1}\mu,\ldots,a_{j-1}\mu^{j-1} \in S$ and $a_{j-1}\notin \overline{S}$, otherwise $a_0 \in \overline{S}$. Note also that this implies that $a_{j-1} \F_{q^t}\cap \overline{S}=\{0\}$. 
We want to prove that $j\leq t$.
By contradiction, assume that $j> t$. It follows that $a_{j-1}\F_{q^t}=a_{j-1}\F_q(\mu)\subseteq S$ and by \eqref{eq:ais} $a_{j-1}\F_{q^t} \subseteq S \cap \mu S \cdots \cap \mu^{j-1}S$, and hence 
\[a_{j-1}\F_{q^t} \oplus \overline{S} \subseteq  S \cap \mu S \cdots \cap \mu^{j-1}S,\] 
i.e.\ $\dim_{\F_q}( S \cap \mu S \cdots \cap \mu^{j-1}S) \geq k-j+t>k-j+1$, which is a contradiction to the minimality of $j$.
Then $j\leq t$ and so $\dim_{\fq}(a_{j-1}\langle 1,\mu,\ldots,\mu^{j-1}\rangle_{\fq})=j$ and 
\[a_{j-1}\langle 1,\mu,\ldots,\mu^{j-1}\rangle_{\fq}\cap \overline{S}\subseteq a_{j-1}\F_{q^t}\cap \overline{S}=\{0\}.\]
Therefore, $S=\overline{S}\oplus a_{j-1}\langle 1,\mu,\ldots,\mu^{j-1}\rangle_{\fq}$. Finally, if $j=t$ then $S=\overline{S} \oplus a_{t-1} \F_{q^t}$ and $S\cap \mathbb{F}_{q^t}=\{0\}$, and so $S=\mu S$, a contradiction. 
\end{proof}

\begin{remark}
\cite[Lemma 4]{BSZ2017} deals with extension fields $L/F$ where $F$ is algebraically closed in $L$, which means that, if $L$ and $F$ are both finite, the degree of the extension $L/F$ is a prime.
In particular, when $n$ is a prime, then Lemma \ref{lemma:power} coincides with \cite[Lemma 4]{BSZ2017}.
\end{remark}

Up to the action of $\mathrm{PGL}(2,q^n)$ on $\PG(1,q^n)$, linear sets admitting two points with complementary weights are of the form described in the following proposition.

\begin{proposition}\label{prop:formls2}
Let $L_W$ be an $\fq$-linear set in $\PG(1,q^n)$ of rank $k \leq n$ admitting two points with complementary weights of type $(k-r,r)$.
Then $L_W$ is $\mathrm{PGL}(2,q^n)$-equivalent to $L_U$ where $U=S\times T$, where $S$ is a $(k-r)$-dimensional $\fq$-subspace of $\fqn$ and $T$ is an $r$-dimensional $\F_q$-subspace of $\F_{q^n}$, with $1 \in T$.
\end{proposition}
\begin{proof}
The proof follows by \cite[Proposition 3.2]{NPSZ2021} and from the fact that if $1 \notin T$ we can replace $T$ by $a^{-1}T$ for some $a \in T\setminus\{0\}$.
\end{proof}

From now on, we assume that the linear sets we are studying  are of the type $L_U$ as in Proposition \ref{prop:formls2}.

\begin{theorem} \label{prop:numberopointsr}
Let $U=S\times T$, where $S$ is a $(k-r)$-dimensional $\fq$-subspace of $\fqn$ with $r \leq k-r$ and $T$ is an $r$-dimensional $\F_q$-subspace of $\F_{q^n}$ and suppose that $T=\langle a_1,\ldots,a_r\rangle_{\fq}$ for some $a_1,\ldots,a_r \in \fqn$.
The set of points of weight $r$ in $L_U$ different from $\langle (1,0)  \rangle_{\fqn}$ is 
\[ \{ \langle (\xi,1)\rangle_{\fqn} \colon \xi \in a_1^{-1}S\cap \ldots \cap a_r^{-1} S \} \]
and its size is $q^j$ with $j=\dim_{\fq}(a_1^{-1}S\cap \ldots \cap a_r^{-1} S)$.
\end{theorem}
\begin{proof}
First observe that, by \eqref{eq:wpointsrank} we have $w_{L_U}(P)\leq r$ for every point $P$ different from $\langle (1,0)  \rangle_{\fqn}$, which has weight $k-r$.
Let consider the map
\[
\begin{tabular}{l c c c }
$\Phi:$ & $a_1^{-1}S \cap \ldots \cap a_r^{-1}S$ & $\longrightarrow$ & $\{P \in L_U \colon w_{L_U}(P)=r, P \neq \langle (1,0) \rangle_{\F_{q^n}}\}$ \\
& $\xi$ & $\longmapsto$ & $P_{\xi}=\langle (\xi,1)\rangle_{\F_{q^n}}$.
\end{tabular}
\]
First, we prove that $\Phi$ is well-defined. Let $\xi \in a_1^{-1}S \cap \ldots \cap a_t^{-1}S$, then there exist $s_1,\ldots,s_r \in S$ such that $\xi=s_i/a_i$, for $i\in\{1,\ldots,r\}$. So the point $P_{\xi}=\langle (\xi,1)\rangle_{\fqn}=\langle (s_i,a_i)\rangle_{\fqn} \in L_U$. Moreover, since $a_1,\ldots,a_t$ are $\F_q$-linearly independent and $s_i/a_i=\xi$ for every $i$, $P_{\xi}$  has weight $r$ in $L_U$. It is clear that $\Phi$ is a one-to-one map. Next we prove that $\Phi$ is surjective. Let $P_{\mu}=\langle (\mu,1) \rangle_{\F_{q^n}} \in L_U$ with $w_{L_U}(P_{\mu})=r$. So there exist $(s_1,t_1),\ldots,(s_r,t_r)\in U$, $\F_q$-linearly independent such that $\langle (s_i,t_i)\rangle_{\F_{q^n}}=P_{\mu}$ and hence $s_i/t_i=\mu$, for $i\in \{1,\ldots,r\}$.
Note that, under these assumptions, $t_1,\ldots,t_r$ are $\F_q$-linearly independent.
Then $\langle t_1,\ldots,t_r \rangle_{\fq}=T$ and hence there exists a matrix $B=(b_{ij}) \in \mathrm{GL}(r,q)$ such that 
\[
B \left( \begin{matrix}
t_1 \\
\vdots \\
t_r
\end{matrix}\right)=\left( \begin{matrix}
a_1 \\
\vdots \\
a_r
\end{matrix}\right).
\]
Let $\overline{s}_1,\ldots,\overline{s}_r \in \fqn$ such that 
\[\left( \begin{matrix}
\overline{s}_1 \\
\vdots \\
\overline{s}_r
\end{matrix}\right)=B \left( \begin{matrix}
s_1 \\
\vdots \\
s_r
\end{matrix}\right).
\]
Hence $\overline{s}_1,\ldots,\overline{s}_r \in S$, and $\sum_{j=1}^{r} b_{ij} (s_j,t_j)=(\overline{s}_i,a_i)$, for $i\in\{1,\ldots,r\}$. Hence $(\overline{s}_i,a_i) \in P_{\mu}\cap U$, for any $i\in\{1,\ldots,r\}$, and $\mu=\overline{s}_i/a_i$. So $\mu \in a_i^{-1}S$, for every $i\in\{1,\ldots,r\}$ and this proves that $\Phi$ is a bijection.
\end{proof}

\subsection{Linear sets with complementary weights of type $(k-2,2)$}

Theorem \ref{prop:numberopointsr} reads as follows in the case of linear sets of rank $k$ with complementary weights of type $(k',2)$, where $k'=k-2$.

\begin{corollary}\label{cor:case(n-2,2)}
Let $U=S\times T$, where $S$ is a $k'$-dimensional $\fq$-subspace of $\fqn$ and $T=\langle 1,\mu\rangle_{\fq}$ for some $\mu \in \fqn\setminus \fq$.
The number of points of weight two in $L_U$ different from $\langle (1,0)  \rangle_{\fqn}$ is $q^j$ with $j=\dim_{\fq}(S\cap \mu S)$.
\end{corollary}

Corollary \ref{cor:case(n-2,2)} allows us to determine the size of a linear set with complementary weights of type $(k',2)$.

\begin{corollary}\label{cor:(n-2,2)2}
Let $U=S\times T$, where $S$ is a $k'$-dimensional $\fq$-subspace of $\fqn$ and $T=\langle 1,\mu\rangle_{\fq}$ for some $\mu \in \fqn\setminus \fq$.
Let $j=\dim_{\fq}(S\cap \mu S)$.
Then
\begin{equation}\label{eq:card2w} 
|L_U|=q^{k'+1}+q^{k'}-q^{j+1}+1. 
\end{equation}
In particular, 
\begin{itemize}
    \item if $j=k'$ then $L_U=L_W$, where $W=\langle U\rangle_{\F_{q^t}}$ and all the point of $L_U$ have weight greater than one. Hence, $\F_{q^t}$ is a geometric field of linearity of $L_U$;
    \item $L_U$ has a point of weight one if and only if $j\leq k'-1$. Also, it has minimum size $q^{k'+1}+1$ if and only if $j=k'-1$.
\end{itemize}
\end{corollary}
\begin{proof}
By \eqref{eq:wpointsrank}, the points of $L_U$ have weight $1$, $2$ or $k'$. If $k'>2$, then $L_U$ contains only one point of weight $k'$ by \eqref{eq:wpointsrank} and $q^j$ points of weight two, by Proposition \ref{prop:numberopointsr}. 
Then \eqref{eq:card2w} follows by \eqref{eq:pesicard} and \eqref{eq:pesivett}.
If $k'=2$, then the number of points of weight two in $L_U$ is $q^j+1$ and we can argue as in the previous case to get \eqref{eq:card2w}.
If $j=k'$, then $S=\mu S$ and by Lemma \ref{lemma:power} $S$ is an $\F_{q^t}$-subspace, with $\F_{q^t}=\F_q(\mu)$. Then $W=\langle S \times T \rangle_{\F_{q^t}}=S \times \F_{q^t}$, and $L_U \subseteq L_W$. Moreover $\lvert L_W \rvert = q^{k'}+1=\lvert L_U \rvert$, so $L_U=L_W$ and an its geometric field of linearity is $\F_{q^t}$.
The last part immediately follows by Equation \eqref{eq:card2w}.
\end{proof}

As a consequence of the above corollary and Lemma \ref{lemma:power} we obtain the following classification of linear sets of minimum size of rank $k=k'+2$ admitting a point of weight $k'$.

\begin{theorem}\label{th:classification2}
Let $L_U$ be an $\fq$-linear set in $\PG(1,q^n)$ of rank $k=k'+2$ with a point of weight $k'$.
Then $L_U$ has minimum size $q^{k'+1}+1$ if and only if, up to the action of $\mathrm{GL}(2,q^n)$, $U=S\times T$,  where $S$ is a $k'$-dimensional $\fq$-subspace of $\fqn$ and $T=\langle 1,\mu\rangle_{\fq}$ for some $\mu \in \fqn\setminus \fq$ with $\fq(\mu)=\F_{q^t}$ and for $S$ one of the following cases occurs:
\begin{enumerate}
    \item $t> k'$ and $S=b \langle 1,\mu,\ldots,\mu^{k'-1}\rangle_{\fq}$, for some $b \in \fqn^*$;
    \item $t\leq k'-1$, write $k'=t\ell+m$ with $m<t$, then $m>0$ and $S=\overline{S}\oplus b\langle 1,\mu,\ldots,\mu^{m-1}\rangle_{\fq}$, where $\overline{S}$ is an $\F_{q^t}$-subspace of dimension $\ell$, $b \in \fqn^*$ and $\overline{S} \cap b\F_{q^t}=\{0\}$.
\end{enumerate}
In particular, if $n$ is prime the only $\fq$-linear sets of rank $k'+2$ in $\PG(1,q^n)$ having minimum size $q^{k'+1}+1$ with a point of weight $k'$ are those described in Theorem \ref{th:constructionVdV}.
\end{theorem}
\begin{proof}
The assertion follows by Proposition \ref{prop:formls2}, Corollary \ref{cor:(n-2,2)2} and
Lemma \ref{lemma:power}.
\end{proof}

\begin{remark} \label{rk:geometricfieldminimum}
Note that in case 2 of Theorem \ref{th:classification2} if $W=\langle U \rangle_{\F_{q^t}}$, then $L_U\subseteq L_W$ and $L_U=L_W$ if and only if $m=t-1$.
\end{remark}

As a corollary we can refine the classification of minimum size linear sets of rank $5$ given in \cite[Proposition 2.12]{DVdV}.

\begin{corollary}
Let $L_U$ be an $\fq$-linear set of rank $5$ with minimum size $q^4+1$ in $\PG(1,q^n)$ containing one point of weight three. 
Then $U$ is $\mathrm{GL}(2,q^n)$-equivalent to one of the following subspace:
\begin{itemize}
    \item $S\times T$, where $S=b \langle 1,\mu,\mu^2,\mu^3,\mu^{4}\rangle_{\fq}$, $T=\langle 1,\mu \rangle_{\fq}$, for some $\mu \in \fqn\setminus \fq$ and $b \in \fqn^*$;
    \item $(\overline{S}\oplus b\fq)\times \F_{q^2}$, where $\overline{S}$ is an $\F_{q^2}$-subspace of dimension $2$, $b \in \fqn$ and $\overline{S} \cap b\F_{q^2}=\{0\}$.
\end{itemize}
In the latter case $n$ is even, $\F_{q^2}$ is a geometric field of linearity of $L_U$,  $L_U=L_W$, where $W=\langle U \rangle_{\F_{q^2}}$, and $L_W$ is a $(\dim_{\F_{q^t}}(W))$-club.
\end{corollary}
\begin{proof}
By applying Theorem \ref{th:classification2} and Remark \ref{rk:geometricfieldminimum}, the desired result is obtained. 
\end{proof}

\subsection{The $n$ prime case}

Next, when $n$ is a prime we give a bound on the number of points of weight $r$. Moreover, if this bound is reached, then both $S$ and $T$ admit a polynomial basis.

\begin{lemma}
 \label{lem:boundpointsr}
Suppose that $n$ is a prime. Let $n\geq k>r\geq 2$ be integers. Let $S$ be a $(k-r)$-dimensional $\fq$-subspace of $\fqn$ with $r \leq k-r$ and let $T$ be an $r$-dimensional $\F_q$-subspace of $\F_{q^n}$ and suppose that $T=\langle a_1,\ldots,a_r\rangle_{\fq}$ for some $a_1,\ldots,a_r \in \fqn$. Let $j=\dim_{\fq}(a_1^{-1}S\cap \ldots \cap a_r^{-1} S)$. Then $j \leq k-2r+1$. Moreover, if $j=k-2r+1 \geq 2$, then $S=g\langle 1,\lambda,\ldots,\lambda^{k-r-1} \rangle_{\F_q}$ and $T=g'\langle  1,\lambda,\ldots,\lambda^{r-1}\rangle_{\F_q}$, for some $g,g',\lambda \in \F_{q^n}$.
\end{lemma}
\begin{proof}
First note that $(a_1^{-1}S\cap \ldots \cap a_r^{-1} S)^{\perp}=(a_1^{-1}S)^{\perp} + \ldots + (a_r^{-1} S)^{\perp}=a_1S^{\perp} + \ldots + a_r S^{\perp}$. This also implies that $\langle S^{\perp}  T \rangle_{\fq}=(a_1^{-1}S\cap \ldots \cap a_r^{-1} S)^{\perp}$ and so $\dim_{\F_q} \langle S^{\perp}  T \rangle_{\fq}=n-j$. Since $n$ is prime, by Theorem \ref{teo:bachocserrazemor}, we have that $\dim_{\F_q} (\langle S^{\perp} T \rangle_{\fq})\geq n-k+2r-1$ and so $j \leq k-2r+1$. If $j=k-2r+1$, then $\dim_{\F_q} (\langle S^{\perp}  T \rangle_{\fq}) = n-k+2r-1$, and  Theorem \ref{teo:bachocserrazemor2} implies that $S^{\perp}=\overline{g}\langle 1,\lambda,\ldots,\lambda^{n-k+r-1} \rangle_{\F_q}$ and $T=g'\langle  1,\lambda,\ldots,\lambda^{r-1}\rangle_{\F_q}$, for some $\overline{g},g',\lambda \in \F_{q^n}$. Therefore $(\overline{g} S)^{\perp}=\overline{g}^{-1}S^{\perp}=\langle 1,\lambda,\ldots,\lambda^{n-k+r-1} \rangle_{\F_q}$, and by Proposition \ref{prop:dualwithdual}, $S=g\langle 1,\lambda,\ldots,\lambda^{k-r-1} \rangle_{\F_q}$, for some $g \in \F_{q^n}$.
\end{proof}

\begin{remark}
For every $n$, the proof of the above lemma can be applied to obtain a lower bound on the number of points of weight $r$. Precisely, a linear set of rank $k$ with a point of weight $k-r$ and a point of weight $r$ (with $r\leq k-r$) contains at least $q^{n-(n-k)r-r^2}$ points of weight $r$.
\end{remark}

\begin{theorem}
Let $n\geq k>r\geq 2$ be integers.
Let $L_U$ be an $\fq$-linear set with complementary weights of type $(k-r,r)$ with a point $P$ of weight $k-r$ and with at least $q^{k-2r}+1$ points different from $P$ of weight $r$.
If $n$ is prime then $L_U$ has minimum size $q^{k-1}+1$ and, up to $\mathrm{G L}(2,q^n)$-equivalence, $L_U$ 
coincides with those described in Theorem \ref{th:constructionVdV}. 
\end{theorem}
\begin{proof}
Suppose that $T=\langle a_1,\ldots,a_r\rangle_{\fq}$ for some $a_1,\ldots,a_r \in \fqn$ and let $j=\dim_{\fq}(a_1^{-1}S\cap \ldots \cap a_r^{-1} S)$. By Theorem \ref{prop:numberopointsr}, $L_U$ has $q^j$ points of weight $r$ different from $P$. Then $q^j\geq q^{k-2r}+1$ and so $j \geq k-2r+1$, and by Lemma \ref{lem:boundpointsr} we have that $j=k-2r+1$ and so  again by Lemma \ref{lem:boundpointsr} the assertion follows.
\end{proof}

However, when $n$ is not a prime, the above classification cannot hold.
We have already seen in Theorem \ref{th:classification2} that there could be more examples of minimum size linear sets with complementary weights of type $(k-2,2)$. In the next sections we will provide more examples of $\fq$-subspaces defining minimum size linear sets not admitting a polynomial basis, and then we will prove that such subspaces can be even not $\mathrm{\Gamma L}(2,q^n)$-equivalent to the one of Theorem \ref{th:constructionVdV}.

\section{Constructions}

We provide new linear sets of minimum size admitting points with complementary weights.
To this aim we describe a procedure which allows us to construct linear sets in $\PG(1,q^n)$ for which its weight spectrum depends on the weight spectrum of a fixed $\fq$-linear set in $\mathrm{PG}(1,q^t)$, with $t \mid n$.

\begin{theorem} \label{th:contructionfromqt}
Let $\ell',t>1$ be positive integers such that $n =\ell' t$. 
Let $U' \neq \{0\}$ be an $\fq$-subspace of $\F_{q^t}^2$ with $\dim_{\fq}(U')=m$ and let $(i_1,\ldots,i_c)$ be the weight spectrum of $L_{U'} \subseteq \PG(1,q^t)$.
Let $\overline{S}$ be an $\F_{q^t}$-subspaces of $\F_{q^n}$ of dimension $\ell<\ell'$ and $b \in \F_{q^n}^*$ such that $\overline{S} \cap b \F_{q^t}=\{0\}$.
%where $U'_1\times\{0\}=U'\cap \langle (1,0) \rangle_{\F_{q^t}}$ and  $\{0\}\times U'_2=U'\cap \langle (0,1) \rangle_{\F_{q^t}}$.
Then the $\fq$-linear set $L_U$ with 
\[U=\{(s+b u_1, u_2) \colon s \in \overline{S}, (u_1, u_2) \in U'\} \] 
has rank $k=\ell t+m$, 
\begin{equation}\label{eq:sizeLUconst}
\lvert L_U \rvert-1 = q^{\ell t} (\lvert L_{U'}  \rvert-1),
\end{equation}
$w_{L_U}(\langle (1,0) \rangle_{\fqn})= t\ell+w_{L_{U'}}(\langle (1,0) \rangle_{\F_{q^t}})$ and all the other points have weight in $\{i_1,\ldots,i_c\}$.
Moreover, denoting by $\overline{N}_{i_h}$ (respectively $N_{i_h}'$) the number of points of $L_U\setminus\{\langle (1,0) \rangle_{\fqn}\}$ (respectively $L_{U'}$) having weight $i_h$ in $L_U$ (respectively in $L_{U'} \setminus \{\langle (1,0) \rangle_{\F_{q^t}}\}$), we have that
\[ \overline{N}_{i_h}=q^{\ell t} N_{i_h}', \]
for any $h \in \{1,\ldots,c\}$.
\end{theorem}
\begin{proof}
If $\overline{S}$ has zero dimension then $L_U=L_{U'}$ and the assertion is trivial. 
Now, suppose that $\overline{S} \neq \{0\}$ and clearly
\[ U=\{(s+b u_1, u_2) \colon s \in \overline{S}, (u_1, u_2) \in U' \}. \]
For any $s \in \overline{S} $ define the map
\[
\begin{array}{lrrl}
    \Psi_{s} \colon & L_{U'} \setminus \{\langle (1,0)\rangle_{\F_{q^n}}\} \subseteq \PG(1,q^n) & \longrightarrow & L_{U} \setminus \{\langle (1,0)\rangle_{\F_{q^n}}\} \subseteq \PG(1,q^n)  \\
    & \langle (u_1,u_2) \rangle_{\F_{q^n}} & \longmapsto & \langle (s+b\frac{u_1}{u_2},1) \rangle_{\F_{q^n}}
\end{array}
\]
Clearly, if $ (u_1,u_2) \in U'$, with $u_2 \neq 0$, then $\langle (s+b\frac{u_1}{u_2},1) \rangle_{\F_{q^n}} \in L_{U} \setminus \{\langle (1,0) \rangle_{\F_{q^n}}\}$ and  $ \langle (u_1,u_2)\rangle_{\F_{q^n}}=\langle (u'_1,u'_2) \rangle_{\F_{q^n}}$ if and only if $\langle (s+b\frac{u_1}{u_2},1) \rangle_{\F_{q^n}}=\langle (s+b\frac{u'_1}{u'_2},1) \rangle_{\F_{q^n}}$, this proves that $\Psi_{s}$ is well-defined and a one-to-one map. Now, let $s_1,s_2 \in \overline{S}$ be such that $s_1\neq  s_2$ we prove that 
\[
\mathrm{Im}(\Psi_{s_1}) \cap \mathrm{Im}(\Psi_{s_2}) = \emptyset,
\]
where $\mathrm{Im}(\Psi_{s_i})=\Psi_{s_i}(L_{U'} \setminus \langle (1,0)\rangle_{\F_{q^t}})$ and $i \in \{1,2\}$. Suppose that $s_1 \neq s_2$ with $s_1,s_2 \in \overline{S}$ and \[\left\langle \left(s_1+b \frac{u_1}{u_2},1\right) \right\rangle_{\F_{q^n}}=\left\langle \left(s_2+b \frac{u'_1}{u'_2},1\right) \right\rangle_{\F_{q^n}},\] for some $ (u_1,u_2), (u'_1,u'_2) \in U'$, with $u_2,u'_2 \neq 0$. This implies that
\[
s_1-s_2=b\left( \frac{u'_1}{u'_2}- \frac{u_1}{u_2} \right).
\]
Since $b$ has been chosen in such a way that $\overline{S}\cap b \F_{q^t}=\{0\}$, we get that $s_1=s_2$, a contradiction.
Now, recalling that $\overline{S}$ is an $\F_{q^t}$-subspace, we observe that any point in $L_U \setminus \{\langle (1,0)\rangle_{\F_{q^n}}\}$ belongs to $\mathrm{Im}(\Psi_{s})$, for some $s \in \overline{S}$, and so the $\mathrm{Im}(\Psi_{s})$'s give a partition of $L_{U}\setminus \{\langle (1,0)\rangle_{\fqn}\}$ in copies of $L_{U'}\setminus \{\langle (1,0)\rangle_{\F_{q^t}}\}$.
Therefore, \eqref{eq:sizeLUconst} is proved.
Let determine the weight distribution of $L_U$. 
Suppose that $\langle (u_1,u_2) \rangle_{\F_{q^t}}$ defines a point of weight $i \in \{i_1,\ldots,i_{c}\}$ in $L_{U'}$, so there exist $(g_h,\overline{g}_h) \in U'$, for $h\in\{1,\ldots,i-1\}$ such that $\langle (u_1,u_2)\rangle_{\F_{q^t}}=\langle(g_h,\overline{g}_h)\rangle_{\F_{q^t}}$ and \[(u_1,u_2),(g_1,\overline{g}_1),\ldots,(g_{i-1},\overline{g}_{i-1})\] are $\F_q$-linearly independent. Let $s \in \overline{S}$, then it follows that 
\[
\left\langle \left(s+b \frac{u_1}{u_2},1 \right) \right\rangle_{\F_{q^n}}=\left\langle \left( s+b \frac{g_h}{\overline{g}_h},1 \right) \right\rangle_{\F_{q^n}}\]  for $h\in \{1,\ldots,i-1\}$ and 
\[ (s u_2+bu_1,u_2),(s\overline{g}_{1}+b g_1, \overline{g}_{1}),\ldots,(s\overline{g}_{i-1}+b g_{i-1},\overline{g}_{i-1})
\]
are $\F_q$-linear independent and so $w_{L_U}(\langle (s+b \frac{u_1}{u_2},1) \rangle_{\F_{q^n}}) \geq i$. Suppose now, that $\langle (s_1+b \frac{u_1}{u_2},1) \rangle_{\F_{q^n}} = \langle (s_2+b \frac{g_i}{\overline{g}_i},1) \rangle_{\F_{q^n}}$. Since if $s_1 \ne s_2$ then $\mathrm{Im}(\Psi_{s_1}) \cap \mathrm{Im}(\Psi_{s_2}) = \emptyset$, it follows that $s_1=s_2$ and $\langle (u_1,u_2)\rangle_{\F_{q^n}}=\langle(g_{i},\overline{g}_{i})\rangle_{\F_{q^n}}$. Hence this implies that $\langle (u_1,u_2)\rangle_{\F_{q^t}}=\langle(g_{i},\overline{g}_{i})\rangle_{\F_{q^t}}$ and so $w_{L_U}(\langle (s+b \frac{u_1}{u_2},1) \rangle_{\F_{q^n}}) = i$. Therefore the number of points of weight $i$ (different from $\langle (1,0) \rangle_{\fqn}$) in $L_U$ is $q^{\ell t}$ times the number of points of weight $i$ (different from $\langle (1,0) \rangle_{\fqn}$) in $L_{U'}$.
\end{proof}

We can now use Theorem \ref{th:contructionfromqt} and Theorem \ref{th:constructionVdV} to obtain new families of minimum size linear sets.

\begin{corollary} \label{th:newminimumsize}
Let $\ell',t>1$ be positive integers such that $n =\ell' t$. Suppose that $\F_{q}(\mu)=\F_{q^t}$, for some $\mu \in \F_{q^n}^*$. Let $\overline{S}$ be an $\F_{q^t}$-subspace of $\F_{q^n}$ of dimension $1 \leq \ell<\ell'$ and $b \in \F_{q^n}^*$ such that $\overline{S} \cap b \F_{q^t}=\{0\}$.
Let $m,j>0$ be integers such that $m+j \leq t+1$. Let $S=\overline{S} \oplus b \langle 1,\mu,\ldots,\mu^{m-1} \rangle_{\F_q}$, $T = \langle 1,\mu,\ldots,\mu^{j-1} \rangle_{\F_q}$ and $U=S \times T$. Then $L_U$ is an $\fq$-linear set of rank $k=\ell t+j+m$ and its size is $q^{k-1}+1$. 
Moreover, $L_U$ has one point of weight $\ell t+m$ and
\begin{itemize}
    \item if $m\geq j$ then it also contains $q^{\ell t+m-j+1}$ points of weight $j$ and $q^{k-2i+1}-q^{k-2i-1}$ points of weight $i$ for $i \in \{1,\ldots,j-1\}$;
    \item if $m<j$ then it also contains $q^{\ell t+j-m+1}-q^{\ell t}$ points of weight $m$, $q^{\ell t}$ points of weight $j$ and $q^{k-2i+1}-q^{k-2i-1}$ points of weight $i$ for $i \in \{1,\ldots,m-1\}$.
\end{itemize}
In particular, $L_U$ has minimum size.
\end{corollary}

If $j>m$ in Corollary \ref{th:newminimumsize}, we obtain linear sets of minimum size with a different weight distribution from those of Theorem \ref{th:constructionVdV}. 

\begin{theorem} \label{th:confrontopesi}
Let $\ell',t>1$ be positive integers such that $n =\ell' t$. Suppose that $\F_{q}(\mu)=\F_{q^t}$, for some $\mu \in \F_{q^n}^*$. Let $\overline{S}$ be an $\F_{q^t}$-subspaces of $\F_{q^n}$ of dimension $\ell$, where $1\leq \ell<\ell'$ and $b \in \F_{q^n}^*$ such that $\overline{S} \cap b \F_{q^t}=\{0\}$.
Let $m,j>0$ be integers such that $m+j \leq t+1$. Let $S=\overline{S} \oplus b \langle 1,\mu,\ldots,\mu^{m-1} \rangle_{\F_q}$, $T = \langle 1,\mu,\ldots,\mu^{j-1} \rangle_{\F_q}$ and $U=S \times T$. 
\begin{itemize}
    \item If $j\leq m$, then there exists a linear set obtained by Construction \ref{th:constructionVdV} with the same weight spectrum and weight distribution of $L_U$.
    \item If $j>m$, then there exists a linear set obtained by Construction \ref{th:constructionVdV} with the same weight spectrum and weight distribution of $L_U$ if and only if $j=m+1$. 
\end{itemize}
\end{theorem}
\begin{proof}
Suppose $j\leq m$. Let $\lambda\in \fqn \setminus \fq$ a generator of $\F_{q^n}$ and consider the linear set
\[ L_W=\{ \langle (\alpha_0+\alpha_1 \lambda+\ldots+\alpha_{t_1-1}\lambda^{t_1-1},\beta_0+\beta_1 \lambda+\ldots+\beta_{t_2-1}\lambda^{t_2-1}) \rangle_{\fqn} \colon \alpha_i,\beta_i \in \fq,\,\]\[\text{not all zero}\}, \]
with $t_1=j$ and $t_2=\ell t+m$. By Theorem \ref{th:constructionVdV}, the linear set $L_W$ has one point of weight $\ell t+m$, $q^{\ell t + m-j+1}$ points of weight $j$ and $q^{\ell t+m+j-2i+1}-q^{\ell t+ m+j-2i-1}$ points of weight $i \in \{1,\ldots, j-1\}$.
Then $L_U$ and $L_W$ have the same weight spectrum and the same weight distribution.
Suppose now that $j>m$ and suppose there exists $\lambda\in \fqn \setminus \fq$ that generates a degree $s$-extension of $\fq$ such that the linear set
\[ L_W=\{ \langle (\alpha_0+\alpha_1 \lambda+\ldots+\alpha_{t_1-1}\lambda^{t_1-1},\beta_0+\beta_1 \lambda+\ldots+\beta_{t_2-1}\lambda^{t_2-1}) \rangle_{\fqn} \colon \alpha_i,\beta_i \in \fq,\,\]\[\text{not all zero}\}, \]
with $t_2 \geq t_1$ and $t_1+t_2 \leq s+1$
has the same weight distribution of $L_U$. In particular, $t_2=\ell t+m$ and $t_1=j$. So, by Theorem \ref{th:constructionVdV}, the linear set $L_W\subseteq \PG(1,q^n)$ has one point of weight $\ell t+m$, $q^{\ell t + m-j+1}$ points of weight $j$ and $q^{\ell t+m+j-2i+1}-q^{\ell t+ m+j-2i-1}$ points of weight $i \in \{1,\ldots, j-1\}$. This implies that $m=j-1$. Indeed, if $m<j-1$, then $L_W$ has a point of weight $j-1$, but $L_U$ does not, a contradiction. So, if $m=j-1$, by Corollary \ref{th:newminimumsize} the weight spectrum and weight distribution of $L_W$, coincides with that of $L_U$.
\end{proof}

\begin{remark}
Theorem \ref{th:confrontopesi} shows that there exist linear sets in Theorem \ref{th:newminimumsize} with a different weight distribution of those in Theorem \ref{th:constructionVdV}, hence the family of Theorem \ref{th:newminimumsize} is larger than the family of Theorem \ref{th:constructionVdV}.
Moreover, this also proves that the weight distribution of a linear set of minimum size with complementary weights of type $(k-r,r)$ is not unique.
\end{remark}

\begin{remark}
In Corollary \ref{th:newminimumsize}, if $W=\langle U \rangle_{\F_{q^t}}$ then clearly $L_U\subseteq L_W$. The equality $L_U=L_W$ holds if and only if $m+j=t+1$. 
In particular, if $L_U$ has rank $n$ the latter case cannot happen.
Indeed, if $m+j=t+1$ then $\ell t+m+j$ should coincide with $n$, but $t\mid n$ and $t \nmid \ell t+m+j$.
\end{remark}

\begin{example}
Consider $\F_{q^4}$, and let $\mu \in \F_{q^4}$ be a generator element, i.e. $\F_{q}(\mu)=\F_{q^4}$.
By Theorem \ref{th:constructionVdV}, the linear set $L_W \subseteq \PG(1,q^4)$, where 
\[
W=\langle 1 \rangle_{\F_q} \times \langle 1,\mu, \mu^2 \rangle_{\F_q},
\]
has weight spectrum $(1,3)$ and minimum size $q^3+1$. Corollary \ref{th:newminimumsize} allows us to \emph{lift} this construction to a linear set of $\PG(1,q^n)$ with minimum size, where $n$ is a multiple of $4$. More precisely, let $\ell'>1$ be positive integer and let $n=4 \ell'$. Let $\overline{S}$ be an $\F_{q^4}$-subspace of $\F_{q^n}$ of dimension $\ell=\ell'-1$ and $b \in \F_{q^n}^*$ such that $b \notin \overline{S}$. Consider \[U=(\overline{S} \oplus b \langle 1 \rangle_{\F_q}) \times \langle 1,\mu,\mu^{2} \rangle_{\F_q} \subseteq \F_{q^n}^2. \] Then, by Corollary \ref{th:newminimumsize} it follows that $L_U$ is an $\fq$-linear set of rank $n$ and size $q^{n-1}+1$. 
Moreover, $L_U$ has weight spectrum $(1,3,n-3)$. In particular $L_U$ has no point of weight $2$. 
Note that this weight spectrum cannot be obtained from the construction in Theorem \ref{th:constructionVdV}.
\end{example}

In the case in which we choose $L_{U'}$ to be a scattered $\fq$-linear set, then the construction of Theorem \ref{th:contructionfromqt} yields to examples of $i$-club.

\begin{corollary}\label{cor:i-clubqt}
Let $\ell',t>1$ be positive integers such that $n =\ell' t$. 
Let $U'$ be an $\fq$-subspace of $\F_{q^t}^2$ with $\dim_{\fq}(U')=m \leq t$ and such that $L_{U'} \subseteq \PG(1,q^t)$ is a scattered $\fq$-linear set.
Let $\overline{S}$ be an $\F_{q^t}$-subspace of $\F_{q^n}$ of dimension $\ell<\ell'$ and $b \in \F_{q^n}^*$ such that $\overline{S} \cap b \F_{q^t}=\{0\}$.
Then the $\fq$-linear set $L_U$ with 
\[U=\{(s+b u_1, u_2) \colon s \in \overline{S}, (u_1, u_2) \in U'\} \]
has rank $k=\ell t+m$, 
\begin{equation}\label{eq:sizeLUconstclub}
\lvert L_U \rvert-1 = q^{\ell t} (q^{t-1}+\ldots+q),
\end{equation}
and $L_U$ is an $i$-club with $i \in \{t\ell,t\ell+1\}$, according to $w_{L_{U'}}(\langle (1,0) \rangle_{\F_{q^t}})\in \{0,1\}$.
\end{corollary}

\begin{remark}
Let $W=\{(x,f(x))\colon x \in \F_{q^t}\}$ be such that $L_{W}$ is a scattered $\fq$-linear set in $\mathrm{PG}(1,q^t)$, for some $q$-polynomial $f(x)$, and let $\{1,\omega,\ldots,\omega^{\frac{n}t-1}\}$ be an $\F_{q^t}$-basis of $\fqn$.
Choosing in Corollary \ref{cor:i-clubqt},
\[U'=\{(f(x)-ax,x)\colon x \in \F_{q^t}\}\,\,\, \text{and}\,\,\, \overline{S}=\langle \omega,\ldots,\omega^{\frac{n}t-1}\rangle_{\F_{q^t}},\]
we obtain the construction of $i$-club provided by De Boeck and Van de Voorde in \cite[Theorem 3.6]{DeBoeckVdV2016}, where these linear sets were associated to the translation KM-arcs constructed by G\'acs and Weiner in \cite{GW03}.
\end{remark}

\section{More on the $\mathrm{\Gamma L}(2,q^n)$-equivalence}

We now show that the two different types of $\fq$-subspaces obtained in Theorem \ref{th:constructionVdV} and Theorem \ref{th:newminimumsize} are $\Gamma\mathrm{L}(2,q^n)$-inequivalent, even if the associated linear sets have the same weight spectrum and the same weight distribution (see Theorem \ref{th:confrontopesi}).

\begin{theorem} \label{th:equivalenceminimumsize}
Let $\ell',t>1$ be two positive integers and let $n =\ell' t$. Suppose that $\F_{q}(\mu)=\F_{q^t}$, for some $\mu \in \F_{q^n}^*$. Let $\overline{S}$ be an $\ell$-dimensional $\F_{q^t}$-subspace of $\F_{q^n}$ with $1\leq \ell<\ell'$ and $b \in \F_{q^n}^*$ such that $\overline{S} \cap b \F_{q^t}=\{0\}$. 
Let $m,j>0$ be integers such that $m+j \leq t+1$. Let $S_1=\overline{S} \oplus b \langle 1,\mu,\ldots,\mu^{m-1} \rangle_{\F_q}$, $T_1 = \langle 1,\mu,\ldots,\mu^{j-1} \rangle_{\F_q}$ and $U_1=S_1 \times T_1$. Suppose that $n=\ell t+m+j$.
Let $\lambda\in \fqn \setminus \fq$ be such that $\fq(\lambda)=\F_{q^s}$ and $\ell t+m+j \leq s+1$. Let $U_2=S_2 \times T_2$, where $S_2=\langle 1,\lambda,\ldots,\lambda^{k-1}\rangle_{\fq}$ and $T_2=\langle 1,\lambda,\ldots,\lambda^{j-1}\rangle_{\fq}$ with $k=\ell t+m$.
Then the $\fq$-subspaces $U_1$ are $U_2$ are $\Gamma\mathrm{L}(2,q^n)$-inequivalent.
\end{theorem}
\begin{proof}
Note that if $j>m-1$, the assertion follows by Remark \ref{rk:inequivalent} and Theorem \ref{th:confrontopesi}. So, let assume that $j \leq m-1$.
First note that since $n=\ell t+m+j$ then $\fq(\lambda)=\fqn$.
Moreover, $n=(\ell+1)t$. Indeed, since $t \mid n$ and $n=\ell t+m+j$ then $t \mid m+j \leq t+1$ and so $m+j=t$. 

Now, suppose by contradiction that $U_1$ and $U_2$ are $\Gamma\mathrm{L}(2,q^n)$-equivalent via $\varphi$.
Since $k>j$, $\langle (1,0) \rangle_{\fqn}$ is the only point in $L_{U_1}$ and in $L_{U_2}$ of weight $k$ and hence $\varphi(U_1 \cap \langle (1,0) \rangle_{\fqn})=U_2\cap \langle (1,0) \rangle_{\fqn}$, that is
\[ aS_1^{\rho}=S_2, \]
for some $a \in \fqn^*$ and $\rho \in \mathrm{Aut}(\fqn)$.
In particular, we have that $a\overline{S}^\rho \subseteq S_2$ and hence $(a\overline{S}^\rho)^\perp \supseteq S_2^\perp$. Note that $\dim_{\F_{q^{t}}}(a\overline{S}^{\rho})=\ell$ and hence $\dim_{\F_{q}}(a\overline{S}^{\rho})=t\ell$, so that $\dim_{\fq}((a\overline{S}^\rho)^\perp)=n-\ell t=t$.
So, $(a\overline{S}^\rho)^\perp$ is an $\F_{q^{t}}$-subspace of $\fqn$ of dimension one.
Consider the ordered basis $\mathcal{B}=(1,\ldots,\lambda^{n-1})$ and its dual basis $\mathcal{B}^*=(\lambda_0^*,\ldots,\lambda_{n-1}^*)$.
It follows that $S_2^\perp=\langle \lambda_{n-j}^*,\ldots,\lambda_{n-1}^* \rangle_{\fq}$ and by Lemma \ref{cor:dualbasis} we have that
\[ \lambda_{n-2}^*=\delta^{-1}(a_{n-1}+\lambda), \]
and
\[ \lambda_{n-1}^*=\delta^{-1}, \]
where $f(x)=a_0+a_1x+\ldots+a_{n-1}x^{n-1}+x^n$ is the minimal polynomial of $\lambda$ over $\fq$ and $\delta=f'(\lambda)$.
Now, since $\lambda_{n-2}^*,\lambda_{n-1}^* \in (a\overline{S}^\rho)^\perp$ and since $(a\overline{S}^\rho)^\perp$ has dimension one over $\F_{q^{t}}$, it follows
\[ \frac{\lambda_{n-2}^*}{\lambda_{n-1}^*}=a_{n-1}+\lambda \in \F_{q^{t}}, \]
that is $\lambda \in \F_{q^{t}}$, a contradiction.
\end{proof}

As a consequence, for $k=n$ and $j=2$ we obtain that the two different types of $\fq$-subspaces obtained in Theorem \ref{th:classification2} are $\Gamma\mathrm{L}(2,q^n)$-inequivalent.

\begin{corollary}
 \label{th:equivalenceminimumsize}
Let $\ell',t>1$ be positive integers such that $n =\ell' t$. Suppose that $\F_{q}(\mu)=\F_{q^t}$, for some $\mu \in \F_{q^n}^*$. Let $\overline{S}$ be an $\ell$-dimensional $\F_{q^t}$-subspace of $\F_{q^n}$ with $1\leq \ell<\ell'$ and let $b \in \F_{q^n}^*$ such that $\overline{S} \cap b \F_{q^t}=\{0\}$. 
Let $m>0$ be an integer with $m \leq t-1$. Let $S_1=\overline{S} \oplus b \langle 1,\mu,\ldots,\mu^{m-1} \rangle_{\F_q}$, $T_1 = \langle 1,\mu \rangle_{\F_q}$ and $U_1=S_1 \times T_1$. Suppose that $n=\ell t+m+2$.
Let $\lambda\in \fqn \setminus \fq$ and let $\fq(\lambda)=\F_{q^s}$ such that $\ell t+m \leq s-1$ and let $U_2=S_2 \times T_2$, where $S_2=\langle 1,\lambda,\ldots,\lambda^{k-1}\rangle_{\fq}$ and $T_2=\langle 1,\lambda \rangle_{\fq}$ with $k=\ell t+m$.
Then the $\fq$-subspaces $U_1$ are $U_2$ are $\Gamma\mathrm{L}(2,q^n)$-inequivalent.
\end{corollary}

The above corollary and Theorem \ref{th:classification2} give a complete answer to the case of minimum size linear sets having two points with complementary weights of type $(k-2,2)$.

\section{Critical pairs}

The linear analogue of Kneser's Addition Theorem \cite{Kneser} has been proved by Hou, Leung and Xiang in \cite{HLX02} for separable extensions fields, motivated by a problem on different sets. 
This latter paper started the study of the linear extensions of classical addition theorems; see e.g. \cite{EL09,L14}. 

Bachoc, Serra and Z\'emor in \cite{BSZ2015} extended the result in \cite{HLX02} by removing the separability condition.

\begin{theorem}\cite[Theorem 3]{BSZ2015} \label{teo:bachocserrazemorext}
Let $L/F$ be a field extension and let $S$ be an $F$-subspace of $L$ of finite positive dimension. Then
\begin{itemize}
    \item either for every finite dimensional $F$-subspace $T \subseteq L$ we have
    \[
    \dim_{F}(\langle S T \rangle_{F})\geq \min\{\dim_{F}(S)+\dim_{F}(T)-1,\dim_{F}(L)\},
    \]
    \item or there exists a subfield $K$ of $L$ with $K \supset F$, such that for every finite dimensional $F$-subspace $T \subseteq L$ satisfying 
    \[
    \dim_{F}(\langle S T \rangle_{F})< \dim_{F}(S)+\dim_{F}(T)-1,
    \]
    we have that $\langle S T \rangle_{F}$ is also a $K$-subspace.
\end{itemize}
\end{theorem}

Theorem \ref{teo:bachocserrazemor} is an instance of this result.
The above result can be also seen as an improvement of the Cauchy-Davenport inequality for field extensions, whose linear analogue was proved in \cite[Theorem 6.2]{EL09} .
Vosper in \cite{Vosper} proved an inverse statement of the Cauchy-Davenport inequality in additive theory by providing examples of pair of sets attaining the equality in the Cauchy-Davenport inequality, which are known as \textbf{critical pairs} (see also \cite{Lev}).
Bachoc, Serra and Z\'emor in Theorem \ref{teo:bachocserrazemor2} provided a linear analogue of Vosper's result in the case of prime degree extensions of finite fields.
As already pointed out by Bachoc, Serra and Z\'emor in \cite[Section 8 (iv)]{BSZ2017}, an important open problem in the field theory regards the characterization of critical pairs achieving the equality in the first point of Theorem \ref{teo:bachocserrazemorext}. 

The subspaces used in Corollary \ref{th:newminimumsize} give examples of critical pairs.

\begin{example}\label{ex:critpair4.2}
Let $L/F$ be a field extension and let $\mu \in L\setminus F$, such that $\mu$ is algebraic over $F$.
Denote $K=F(\mu)$ and $t=\dim_F(K)$.
Consider $\overline{S}$ a $K$-subspace of $L$ of dimension $\ell$, with $1 \leq \ell<\dim_K(L)$, and $b \in L^*$ such that $\overline{S} \cap b K=\{0\}$.
Let $m,j>0$ be integers such that $m+j \leq t+1$. 
It is easy to verify that the couple $(S,T)$, where $S=\overline{S} \oplus b \langle 1,\mu,\ldots,\mu^{m-1} \rangle_{F}$ and $T = \langle 1,\mu,\ldots,\mu^{j-1} \rangle_{F}$, is a critical pair.
Indeed,
\[ \langle ST\rangle_F=\overline{S}\oplus b\langle 1,\mu,\ldots,\mu^{m+j-2} \rangle_{F}, \]
which has dimension $\ell+m+j-1$ over $F$.
\end{example}

Lemma \ref{lemma:power} can be adapted to prove a characterization of critical pairs $(S,T)$ when considering extension fields $L/F$ in which $T$ is a two-dimensional $F$-subspace of $L$ which, up to a scalar multiple, is generated by $1$ and an algebraic element of $L$ over $T$, which was known only in the case of prime degree extensions of finite fields.

\begin{proposition} \label{prop:extalg}
Let $L/F$ be a field extension and let $S$ be an $F$-subspace of $L$ of finite positive dimension $k$ and let $T=\langle 1,\mu\rangle_F$, for some $\mu \in L\setminus F$, such that $k+2 \leq \dim_F(L)$ and $\mu$ is algebraic over $F$.
Let $K=F(\mu)$ and let $t=\dim_F(K)$.
Suppose that 
\[\dim_{F}(\langle S T \rangle_{F})=\dim_F(S)+1.\]
Then one of the following cases occurs:
\begin{itemize}
    \item $t> k$ and $S=b \langle 1,\mu,\ldots,\mu^{k-1}\rangle_{F}$, for some $b \in L^*$;
    \item $t\leq k-1$, write $k=t\ell+m$ with $m<t$, $m>0$ and $S=\overline{S}\oplus b\langle 1,\mu,\ldots,\mu^{m-1}\rangle_{F}$, where $\overline{S}$ is a $K$-subspace of dimension $\ell$, $b \in L^*$ and $b K \cap \overline{S}=\{0\}$.
\end{itemize}
\end{proposition}
\begin{proof}
First note that
\[ \dim_F(\langle ST \rangle_F)= \dim_F(S+\mu S)=k+1, \]
and hence by Grassmann's formula 
\[ \dim_F(S\cap \mu S)=k-1. \]
Now, arguing as in the proof of Lemma \ref{lemma:power} one gets the desired result.
\end{proof}

\begin{example}
Let $\mathbb{R}$ be the real number field and let $\mathbb{Q}$ be the rational field. 
Consider $T=\langle 1,\sqrt[3]{2} \rangle_{\mathbb{Q}}$ and $S$ a $k$-dimensional $\mathbb{Q}$-subspace of $\mathbb{R}$ with $k \geq 4$.
If $\dim_{\mathbb{Q}}(\langle ST \rangle_{\mathbb{Q}})=k+1$, then by Proposition \ref{prop:extalg}, it follows that $k=3\ell+m$ with $\ell \in \mathbb{Z}^+$ and $m \in\{1,2\}$. 
Accordingly to $m$, we have two possibilities:
\begin{itemize}
    \item if $m=1$ then $S=\overline{S}+\langle b\rangle_{\mathbb{Q}}$;
    \item if $m=2$ then $S=\overline{S}+b\langle 1, \sqrt[3]{2}\rangle_{\mathbb{Q}}$, 
\end{itemize}
where $\overline{S}$ is an $\ell$-dimensional $\mathbb{Q}(\sqrt[3]{2})$-subspace and $b \in \mathbb{R}\setminus \overline{S}$.
By Example \ref{ex:critpair4.2}, the couple $(S,T)$ is a critical pair, and it's clear that $S$ does not admit a polynomial basis in $\sqrt[3]{2}$ when $\ell>2$.
\end{example}

Moreover, as a consequence of Theorem \ref{prop:numberopointsr} the critical pairs over finite fields are also related to linear sets of special type.

\begin{proposition}\label{prop:critpairs}
Let $S,T$ be two $\fq$-subspaces of $\fqn$ having dimension $n-k+r$ and $r$, respectively, with $2r\leq k\leq n+r$ and let $U=S^\perp \times T$.
Then $(S,T)$ is a critical pair if and only if the number of points of weight $r$ in $L_U$, different from $\langle (1,0)\rangle_{\fqn}$, is $q^{k-2r+1}$.
\end{proposition}
\begin{proof}
Let $T=\langle a_1,\ldots,a_r\rangle_{\fq}$, then by Theorem \ref{prop:numberopointsr}, the number of points of weight $r$ in $L_U$ is $q^j$, where $j=\dim_{\fq}(a_1^{-1}S^\perp\cap\ldots\cap a_r^{-1}S^\perp)$.
As in the proof of Lemma \ref{lem:boundpointsr}, we note that
\[ \langle S T\rangle_{\fq}=(a_1^{-1}S^\perp\cap\ldots\cap a_r^{-1}S^\perp)^\perp. \]
So, $j=n-\dim_{\fq}(\langle S T\rangle_{\fq})$, and hence $j=k-2r+1$ if and only if $\dim_{\fq}(\langle S T\rangle_{\fq})=n-k+2r-1$, i.e.\ $(S,T)$ is a critical pair.
\end{proof}

Critical pairs could be good candidates to define eventually new examples of linear sets of minimum size having points with complementary weights in the case in which $n$ is not a prime. However, it is not clear to us whether or not there exist critical pairs such that the associated linear sets as in Proposition \ref{prop:critpairs} are not of minimum size.

\begin{small}
%\section*{Acknowledgements}
%The research was supported by the project ``VALERE: VAnviteLli pEr la RicErca" of the University of Campania ``Luigi Vanvitelli'' and was partially supported by the Italian National Group for Algebraic and Geometric Structures and their Applications (GNSAGA - INdAM).

\end{small}

\end{document}